\documentclass[reqno, 11pt]{amsart}

\usepackage{setspace}
\usepackage{biblatex}
\usepackage{amsmath} 
\usepackage{amsthm}
\usepackage{amssymb}
\usepackage{hyperref}
\usepackage{color}
\usepackage{tikz}
\usepackage{setspace}
\usepackage{graphicx}
\usepackage{subcaption}
\usepackage{caption}
\usepackage{changepage}

\calclayout

\newtheorem{theorem}{Theorem}[section]
\newtheorem{lemma}[theorem]{Lemma}
\newtheorem{proposition}[theorem]{Proposition}
 
\newtheorem{corollary}[theorem]{Corollary}
\theoremstyle{definition}
\newtheorem{definition}[theorem]{Definition}
\newtheorem{example}[theorem]{Example}
\theoremstyle{remark}
\newtheorem{remark}[theorem]{Remark}

\newcommand{\mA}{\mathcal{A}}
\newcommand{\mB}{\mathcal{B}}

\newcommand{\mF}{\mathcal{F}}
\newcommand{\mU}{\mathcal{U}}
\newcommand{\mT}{\mathcal{T}}

\newcommand{\mP}{\mathcal{P}}
\newcommand{\mH}{\mathcal{H}}
\newcommand{\mI}{\mathcal{I}}
\newcommand{\mJ}{\mathcal{J}}

\newcommand{\bS}{\mathbf{S}}
\newcommand{\bSG}{\mathbf{S}_G}
\newcommand{\ra}{\rightarrow}
\newcommand{\sli}{\sum\limits}
\newcommand{\schild}[2]{{#1}\text{-child}({#2})}

\newcommand{\al}{\alpha}

\newcommand{\R}{\mathbb{R}}
\newcommand{\sse}{\subseteq}
\newcommand{\nsse}{\nsubseteq}

\newcommand{\cadet}{\text{cadet}}
\newcommand{\parent}{\text{parent}}
\newcommand{\lsib}{\text{lsib}}

\DeclareFieldFormat{doi}{\href{https://dx.doi.org/#1}{\textsc{doi}}}
\DeclareFieldFormat{url}{\href{#1}{Link}}

\title[counting formula for braid-type arrangement]{The geometry of a counting formula for deformations of the braid arrangement}

\author{Neha Goregaokar}
\address{Brandeis University}
\email{ngoregaokar@brandeis.edu}
\author{Aaron Z. Lin}
\address{Ladue Horton Watkins High School}
\email{aaronlin0924@gmail.com}
\date{\today}

\begin{document}
\onehalfspacing

\begin{abstract}

We consider real hyperplane arrangements whose hyperplanes are of the form $\{x_i - x_j = s\}$ for some integer $s$, which we call \textit{deformations of the braid arrangement}. In 2018, Bernardi gave a counting formula for the number of regions of any deformation of the braid arrangement $\mathcal{A}$ as a signed sum over some decorated trees. He further showed that each of these decorated trees can be associated to a region $R$ of the arrangement $\mathcal{A}$, and hence we can consider the contribution of each region to the signed sum. Bernardi also implicitly showed that for \textit{transitive} arrangements, the contribution of any region of the arrangement is $1$. 
We remove the transitivity condition, showing that for \textit{any} deformation of the braid arrangement the contribution of a region to the signed sum is $1$. 

This provides an alternative proof of the original counting formula, and sheds light on the geometry underlying the formula. We further use this new geometric understanding to better understand the contribution of a tree. 
\end{abstract}

\maketitle

\section{Introduction}

A real hyperplane arrangement is a finite collection of affine hyperplanes in a vector space $\mathbb{R}^n$, which decompose the space into polyhedral regions. The study of hyperplane arrangements lies at the intersection of combinatorics, algebraic geometry, and topology. By examining the combinatorial and geometric structure of the regions and their intersections, researchers have uncovered connections to optimization, machine learning, and theoretical physics, where hyperplanes often represent decision boundaries or phase transitions \cite{orlik1992arrangements, POSTNIKOV2000544, stanley2007hyperplane}.

A central question in the study of hyperplane arrangements is: {\it how many distinct regions are formed by a given arrangement?} This question lies at the heart of the subject because the number of regions encodes essential topological and combinatorial information about the arrangement. For example, Zaslavsky’s theorem expresses the number of regions as an evaluation of the characteristic polynomial of the arrangement’s intersection lattice \cite{zaslavsky1975facing}. Region counts also arise in applications ranging from optimization and computational geometry \cite{stanley2007hyperplane} to theoretical physics and machine learning \cite{orlik1992arrangements, POSTNIKOV2000544}.

For a well-studied family of arrangements known as \emph{deformations of the braid arrangement}, where each hyperplane has the form $x_i - x_j = s$ for some integer $s$, this question has been addressed by a counting formula given by Bernardi in \cite{BERNARDI2018466}. The formula is a signed sum on certain decorated plane trees called \textit{boxed trees}. Bernardi's work provided a unified combinatorial interpretation for classical families including the braid, Catalan, Shi, semiorder, and Linial arrangements. 

Further, in~\cite{BERNARDI2018466}, Bernardi proved a bijection (see Theorem~\ref{CatalanRegionBijection}) that gives us a natural way to restrict this signed sum to a region of the arrangement, obtaining the \textit{contribution of a region}. Hence, the question arises: for an arbitrary deformation of the braid arrangement, what is the contribution of a region? 

In~\cite{BERNARDI2018466}, Bernardi shows that for deformations of the braid arrangement which satisfy an extra condition known as \textit{transitivity}, the contribution of any region is $1$. In this paper, we answer the above question in full generality, showing that the contribution of any region of a deformation of the braid arrangement is $1$. Hence, we obtain the original Bernardi formula as an easy consequence of our result. This result not only broadens the applicability of Bernardi’s combinatorial machinery but also deepens our understanding of the topological structure underlying general hyperplane arrangements.

The paper is organized as follows. In Section~\ref{bgsec}, we introduce deformations of the braid arrangement. We further recall the Bernardi formula and define the contribution of a region. In Section~\ref{GASec}, we look at \textit{graphical arrangements}. We establish a bijection between the decorated trees in the signed sum and certain faces of the braid arrangement. Using this bijection, we show that the contribution of a region is in fact its Euler characteristic, which is known to be $1$. In Section~\ref{GenSec}, we generalize this result to all deformations of the braid arrangement via a similar bijection. Finally, in Section~\ref{sec:contribution}, we further restrict the signed sum to a tree and consider the \emph{contribution of a tree}. We use the bijection with faces to better understand and calculate this value. 

\section{Background}\label{bgsec}
In this section, we briefly review Bernardi's formula. First, we provide some background definitions and results on hyperplane arrangements. For non-negative integers $m$ and $n$, we use the following notation $[-m;n]:= \{-m, \ldots, n\}$.

\subsection{Hyperplane Arrangements}
\hfill\\

We begin with a fundamental hyperplane arrangement, the braid arrangement. 

\begin{definition}
    The \emph{braid arrangement} $\mB_n$ in $\mathbb{R}^n$ is the collection of hyperplanes $$\mathcal{B}_n := \{ H_{a,b} \mid 1 \leq a< b \leq n \},$$ where 
     $H_{a,b} = \{ (x_1, \ldots, x_n) \in \R^n \mid x_a - x_b = 0\}.$
\end{definition}

\begin{figure}[h]
\centering
\begin{tikzpicture}[scale=1.5]
\draw[thick,<-] (-1.732, -1) -- (1.732, 1);
\node at (-1.9, 1.1) {$x_1=x_2$};
\node at (-1.9, -1) {$x_2$};
\draw[thick,->] (0, -1.5) -- (0, 1.5);
\node at (0.2,1.4){$x_1$};
\node at (0.5, -1.4) {$x_2=x_3$};
\draw[thick,->] (-1.732, 1) -- (1.732, -1);
\node at (1.9, 1.1) {$x_1=x_3$};
\node at (1.9, -1) {$x_3$};
\end{tikzpicture}
\caption[]{Braid arrangement for $n=3$. Each line corresponds to an equation of the form $x_i=x_j$. \footnotemark}\label{fig:braid}
\end{figure}
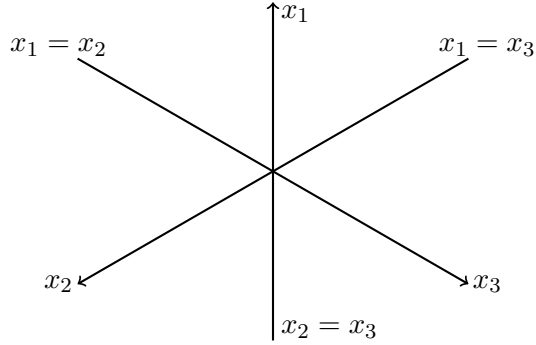

\footnotetext{Note that the hyperplane \( H = \{(x_1, x_2, x_3) \in \mathbb{R}^3 \mid x_1 + x_2 + x_3 = 0\} \) is orthogonal to every hyperplane in \( \mathcal{B}_3 \). Thus, we visualize \( \mathcal{B}_3 \) by depicting the intersections of its hyperplanes with \( H \).}

Figure~\ref{fig:braid} depicts the braid arrangement $\mB_3$. While it is easy to see that the number of regions of the braid arrangement is $n!$, the problem becomes much more challenging for generalizations of the braid arrangement $\mathcal{B}_n$ where we consider hyperplanes of the form of $x_a-x_b = s$, where $s$ is an integer. 

\begin{definition}
    Let $\bS = (S_{a,b})_{1 \leq a < b \leq n}$ be a collection of finite sets of integers. We define the \emph{$\bS$-braid arrangement}, denoted $\mathcal{A}_{\bS}$, as the collection of the following hyperplanes:
    $$\mA_{\bS} = \{H_{a,b,s} \mid 1 \leq a < b \leq n, s \in S_{a,b}\},$$ where
    $H_{a,b,s} = \{ (x_1, \ldots, x_n) \in \R^n \mid x_a - x_b = s,\, 1 \leq a<b\leq n, \, s \in S_{a,b} \}.$
    We collectively call such arrangements \emph{deformations of the braid arrangement}. 
\end{definition}

For $1 \leq a < b \leq n$, we also define the following sets related to $S_{a,b}$: $$S^{-}_{a,b} = \{ s \geq 0 \mid -s \in S_{a,b} \} \mbox{ and }
S^{-}_{b,a} = \{ s > 0 \mid s \in S_{a,b} \} \cup \{0\}.$$

A well studied example of $\bS$-braid arrangement is the $m$-Catalan arrangement, where $S_{a,b} = [-m;m]$ for all $1 \leq a < b \leq n$. The number of regions formed by the \( m \)-Catalan arrangement $\mathcal{C}_n^{(m)}$ is given by
$n^{-1} \binom{(m+1)n}{n - 1}n!$, which is also equal to the number of labeled $(m+1)$-ary trees. Bernardi's formula, introduced in the next subsection, gives us this count. 
This result reflects the deep combinatorial structure of these arrangements. For further details, see \cite{POSTNIKOV2000544}. The counting problem for a more general collection of sets $\bS$ is far more challenging.

\subsection{The Bernardi Formula}
\hfill\\

Next, we describe the combinatorial framework underlying the Bernardi formula, which relies on special labeled rooted plane trees called \textit{$\bS$-boxed trees}.

We first establish some notations. We denote by $\mathcal T$ the set of rooted plane trees with labeled nodes, and $\mathcal{T}^{(m)}(n)$ the set of $(m+1)$-ary rooted plane trees with $n$ labeled nodes in $\mathcal T$. By definition, a node $j \in [n]$ of $T \in \mathcal{T}$ has exactly $m+1$ (ordered) children, which are denoted by $\schild{0}{j}$, $\schild{1}{j}, \ldots, \schild{m}{j}$ respectively. Further, if $v$ is not the root node, we denote by $\parent(v)$ the parent of $v$ and $\lsib(v)$ the number of siblings of $v$ to the left of $v$.

\begin{figure}[h]
    \centering
    \includegraphics[width=0.45\linewidth]{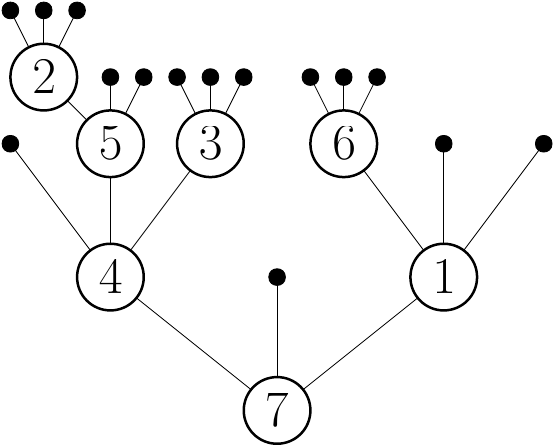}
    \caption{A tree in $\mT^{(2)}(7)$.}
    \label{fig:treeexample}
\end{figure}

\begin{definition}\label{cadetseqdef}
    Let $T$ be a tree in $\mT$.
    \begin{itemize}
        \item Let $u$ be a node of $T$. The \emph{cadet} of $u$ is the rightmost child of $u$ that is not a leaf, if such a child exists. We denote this node as $\cadet(u)$. 
        \item A \emph{cadet sequence} is a nonempty sequence of nodes $ (v_1, v_2, \dots, v_k) $ such that for all $ i \in \{1, \dots, k-1\} $, we have $ v_{i+1} = \cadet(v_i) $.
    \end{itemize}
\end{definition}

\begin{example}
    In Figure~\ref{fig:treeexample}, $\cadet(7) = 1$, $\cadet(1) = 6$, $\cadet(4) = 3$ and $\cadet(5) = 2$. The nodes $2, 3$ and $6$ do not have a cadet as all their children are leaves. Further, $(7,1,6)$, $(7,1)$, $(1,6)$, $(4,3)$, $(5,2)$, as well as $(i)$ for all $i \in [7]$ are all the cadet sequences.
\end{example}

\begin{definition}\label{scadetseqdef}
     An \emph{$ \bS$-cadet sequence} is a cadet sequence $(v_1, \ldots, v_k)$ such that for all $1 \le i < j \le k$, $\sum_{p=i+1}^j \mathrm{lsib}(v_p) \notin S^-_{v_i, v_j}$.
\end{definition}

\begin{definition}\label{boxedtree}
An \emph{$ \bS $-boxed tree} is a pair $ (T, B) $, with $ T\in \mathcal{T}^{(m)}(n)$, where $m = \max(|s|, s \in S_{a,b}, 1 \leq a < b \leq n)$ and $ B $ is a set of $ \bS $-cadet sequences that partition the set of nodes of $ T $. That is, every node of $ T $ appears in exactly one cadet sequence in $ B $. Each element of $B$ is called a \emph{box}. 
\end{definition}

\begin{figure}[h]
    \centering
    \includegraphics[width=0.45\linewidth]{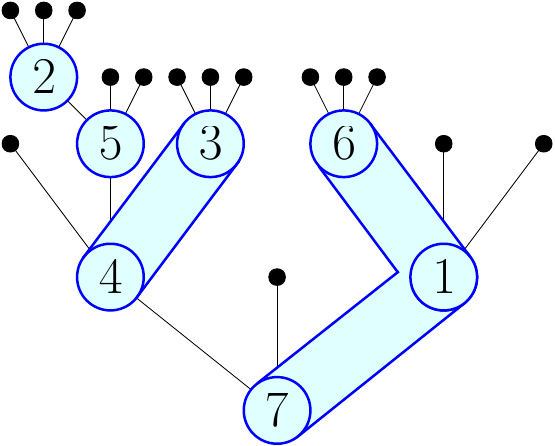}
    \caption{An example for an $\bS$-boxed tree for $S_{a,b} = \{-2,1\}$ for all $a<b$. }
    \label{fig:scadet}
\end{figure}

\begin{example}
    Figure~\ref{fig:scadet} is an example of an $\bS$-boxed tree $(T, B)$ for $\bS = (S_{a,b})_{1 \leq a < b \leq 7}$ with $S_{a, b} = \{-2,1\}$ for all $1 \leq a < b \leq 7$. The $\bS$-cadet sequences in $B$ are denoted by nodes connected by continuous regions of light blue. Specifically, $B = \{(2), (5), (4,3), (7,1,6)\}$. 

    Now, for all $a < b$, $S_{a,b}^- = \{2\}$ and $S_{b,a}^- = \{0,1\}$. Also, $\lsib(1) = 2$, $\lsib(6) = 0$ and $\lsib(3) = 2$. Clearly, $\lsib(3) \notin S_{4,3}^-$, so $(4,3)$ is an $\bS$-cadet sequence. Further, as $\lsib(1) \notin S_{7,1}^{-}$, $\lsib(6) \notin S_{1,6}^{-}$ and $\lsib(1) + \lsib(6) \notin S_{7,6}^{-}$, $(7,1,6)$ is an $\bS$-cadet sequence. 

    Moreover, note that although $(5,2)$ is a cadet sequence, as $\lsib(2) = 0 \in S_{5,2}^-$, $(5,2)$ is \textit{not} an $\bS$-cadet sequence.
\end{example}

We now state the Bernardi formula for the number of regions of an $\bS$-braid arrangement.
\begin{theorem}[Bernardi Formula for $ \bS $-braid arrangements {\cite[Theorem 4.2]{BERNARDI2018466}}] \label{bernardi}

\noindent Let $ \bS = (S_{a,b})_{1 \leq a < b \leq n} $ be a collection of finite sets of integers, and let $ \mA_{\bS}$ be the $\bS$-braid arrangement in $\R^n$.
The number of regions of $ \mathcal{A}_{\bS} $ is given by
$$
r_{\bS}(n) = \sum_{(T,B) \in \mU_{\bS}(n)} (-1)^{n - |B|},
$$
where $|B|$ is the number of boxes in the $\bS$-boxed tree $(T, B)$ and $\mU_{\bS}(n)$ denotes the set of $\bS$-boxed trees with $n$ nodes. 
\end{theorem}

\subsection{Restriction to a region}
\hfill\\

Bernardi defined a total order $\prec_T$ on the vertices of a tree $T$ in $\mT^{(m)}(n)$ (see Section 8.1 of~\cite{BERNARDI2018466} for a precise definition). Using this ordering, he proved the following bijection between trees in $\mT^{(m)}(n)$ and regions of the $m$-Catalan arrangement, that is, the $\bS$-braid arrangement with $S_{a,b} = [-m; m]$ for all $1 \leq a < b \leq n$. 
\begin{theorem}[{\cite[Theorem 8.8]{BERNARDI2018466}}]\label{CatalanRegionBijection}
    Let $T \in \mT^{(m)}(n)$ be a tree. We define the corresponding polyhedron $\Phi(T) \subseteq \R^n$ by: $$\Phi(T) = \left(\bigcap_{\substack{(i,j,s) \in \text{Triple}_n^m \\ i \prec_T \schild{s}{j}}} \{x_i - x_j < s\} \right) \cap \left(\bigcap_{\substack{(i,j,s) \in \text{Triple}_n^m \\ i \succeq_T \schild{s}{j}}} \{x_i - x_j > s\}\right) $$
    where $Triple_n^m = \{(i,j,s) \mid i, j \in [n], i \neq j, \, s \in [0; m] \text{ such that } s > 0 \text{ or } i>j\}. $ Then, $\Phi$ defines a bijection between $\mT^{(m)}(n)$ and regions of the $m$-Catalan arrangement. 
\end{theorem}

Figure \ref{fig:catalan_bijection} illustrates the bijection for the 1-Catalan arrangement.

\begin{figure}
    \centering
    \includegraphics[width=\linewidth]{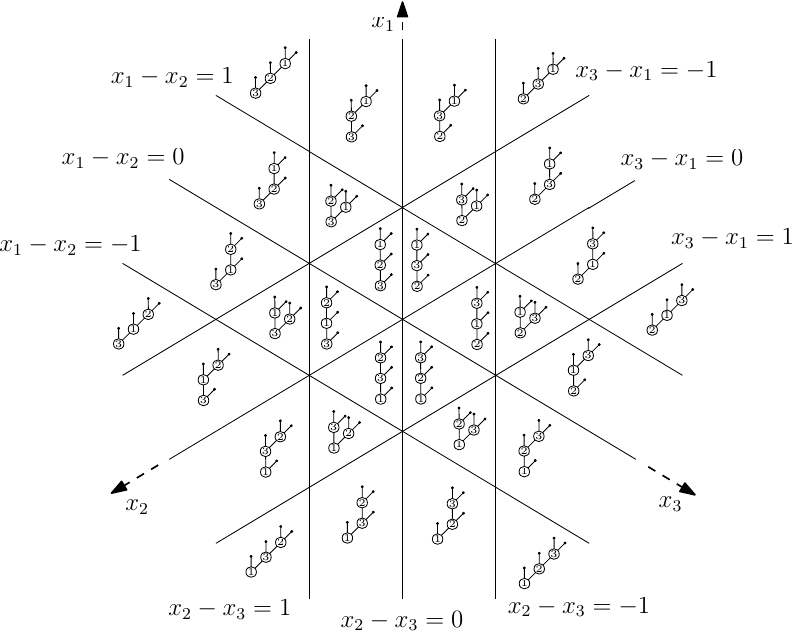}
    \caption{Bijection between the set of trees in $\mT^{(1)}(3)$ and the regions of the 1-Catalan arrangement corresponding to each tree.}
    \label{fig:catalan_bijection}
\end{figure}

Note that an $\bS$-braid arrangement is a subarrangement of the $m$-Catalan arrangement where $m = \max(|s|\, , \, s \in S_{a, b}, \, 1 \leq a < b \leq n)$. Hence, any region of the $\bS$-braid arrangement is a union of regions of the $m$-Catalan arrangement. By Theorem~\ref{CatalanRegionBijection}, any region $R$ of the $\bS$-braid arrangement is associated to a set of trees in $\mT^{(m)}(n)$. We denote this set by $T_R$.

\begin{definition}
    Let $R$ be a region of the $\bS$-braid arrangement. We define 
    $$U_{\bS}(R) : = \{(T, B) \mid T \in T_R, \, B \text{ is an $\bS$-boxing of $T$}\},$$ that is, $U_{\bS}(R)$ is the set of $\bS$-boxed trees where the underlying tree is associated to the region $R$. 
\end{definition}

\section{The Contribution of a Region}

In this paper, we show that the following region-wise refinement of the Bernardi formula holds for all deformations of the braid arrangement.
\begin{theorem}\label{generalcase}

Let $\bS = (S_{a,b})_{1 \leq a < b \leq n} $ be a collection of finite sets of integers and let $ \mathcal{A}_S $ be the corresponding $\bS$-braid arrangement. Let $R$ be a region of $\mA_{\bS}$. Then,
\begin{equation}\label{eqn:bernardi_region}
    \sum_{(T,B) \in \mU_{\bS}(R)} (-1)^{n-|B|} = 1.
\end{equation}
\end{theorem}

\begin{remark}
    Note that for a $\bS$-braid arrangement, $\mA_{\bS}$, by the above theorem, we have   
    \begin{align*}
        \sum_{(T,B) \in \mU_{\bS}(n)} (-1)^{n - |B|} &= \sli_{R \text{ region of } \mA_{\bS}}\sum_{(T,B) \in \mU_{\bS}(R)} (-1)^{n-|B|} \\&= \sli_{R \text{ region of } \mA_{\bS}} 1 \\&= \# \text{ regions of } \mA_{\bS}.
    \end{align*}

    Hence, the Bernardi formula (Theorem~\ref{bernardi}) is a direct consequence of Theorem~\ref{generalcase}. 
\end{remark}

In~\cite{BERNARDI2018466}, Bernardi showed that Equation~\eqref{eqn:bernardi_region} holds for \textit{transitive} arrangements, that is arrangements which satisfy the condition that if $s \notin S_{a,b}^{-}$ and $t \notin S_{b ,c}^{-}$, then $s+t \notin S_{a,c}^{-}$. 

We start in Section~\ref{GASec} by proving Theorem~\ref{generalcase} for graphical arrangements, a very common group of non-transitive arrangements, and continue in Section~\ref{GenSec}, where we extend our proof to show that it holds for all general deformations of the braid arrangement.

\smallskip
\subsection{Graphical Arrangements}\label{GASec}\hfill\\

A graphical arrangement is a subarrangement of the braid arrangement, where the subset of hyperplanes is determined by the adjacency structure of a graph. In this paper, we use \emph{graph} to mean an undirected finite graph without loops or multiple edges. 

\begin{definition}
Let $G = ([n], E)$ be a graph. The \emph{graphical arrangement} associated with $G$, denoted $\mA_G$ is the collection of the hyperplanes $$\mathcal{A}_G := \{H_{a,b} \mid\{a,b\}\in E\},$$ where $ H_{a,b} = \{ (x_1, \ldots, x_n) \in \R^n \mid x_a - x_b = 0\}.$ 
\end{definition}

\begin{remark}
Let $G = ([n], E)$ be a graph. For $a< b$, define $S_{a, b} = \begin{cases}
     \{0\} & \text{ if } \{a,b\} \in E \\
     \emptyset & \text{ otherwise}
 \end{cases}.$ Then, for $\bS_G = (S_{a,b})_{1\leq a < b \leq n}$, the graphical arrangement $\mA_G$ is the $\bS_G$-braid arrangement. 
\end{remark}

We next prove that the region-wise Bernardi formula, that is, Equation~\eqref{eqn:bernardi_region}, holds for any region in a graphical arrangement. The key to our proof is a bijection between $\bS_G$-boxed trees and certain faces of the braid arrangement. This then equates the signed sum in Equation~\eqref{eqn:bernardi_region} to the Euler characteristic of a region, which is always equal to 1 (as regions of hyperplane arrangements are contractible). 

\begin{remark}\label{rem:GBoxedTrees}
    Note that for graphical arrangements, $m = \max(|s|, s \in S_{a,b}, 1 \leq a < b \leq n) = 0$. Hence, $\mT^{(m)}(n) = \mT^{(0)}(n)$, the set of labeled unary trees. Further, any labeled unary tree can be uniquely represented as $T = v_1 \ldots v_n$ where $v_1$ is the root of $T$ and $v_{i+1}$ is the child of $v_i$ for all $i \in [n-1]$. Here, $\{v_1, \ldots, v_n\} = [n]$.

    Hence, for a $\bS_G$-boxed tree $(T, B) \in \mU_{\bS_G}(n)$, with $T = v_1\ldots v_n$ as above, $B$ will be of the form $\{(v_1, \ldots, v_{i_1}), (v_{i_1 + 1}, \ldots, v_{i_2}) \ldots, (v_{i_k + 1}, \ldots, v_n)\}$, where for all $0 \leq j \leq k$, $(v_{i_j + 1}, \ldots, v_{i_{j+1}})$ is an $\bS_G$-cadet sequence (with $i_0 = 0$, $i_{k+1} = n$). 
\end{remark}

Note that the maximum absolute intercept $m$ is 0 in a graphical arrangement, so it follows from Definition~\ref{boxedtree} that the $\bS_G$-boxed trees are all unary trees (they identify with paths). The following lemma further characterizes the $\bS_G$-boxed trees.  
\begin{lemma}\label{graphcadet}
     A cadet sequence $(v_1, \ldots, v_k)$ is an $\bS_G$-cadet sequence for $\mA_G$ if and only if for all $1 \leq i < j \leq k$, we have $v_i < v_j$ and $\{v_i, v_j\} \notin E$. 
\end{lemma}
\begin{proof}
    As $\bS_G$-boxed trees are unary, we have $\sli_{p = i+1}^j \lsib(v_p) = 0$ for all $1 \leq i < j \leq n$. The result follows by observing that $$0 \notin S_{v_i, v_j}^{-} \iff v_i < v_j \text{ and } \{v_i, v_j\} \notin E. $$
\end{proof}

Next, we show that these $\bSG$-boxed trees are in bijection with certain faces of the braid arrangement.

\begin{definition}
    Let $\mF_{\mB_n}$ be the set of faces of the braid arrangement $\mB_n$. We define $$\mF_G:= \{F \in \mF_{\mB_n} \mid \forall H \in \mA_G, \,\, F \nsubseteq H\},$$
    that is, $\mF_G$ is the set of faces of the braid arrangement that are not contained in any hyperplane of the graphical arrangement $\mA_G$. 
    
    Further, for a region $R$ of $\mA_G$, we define $\mF_R: = \{F \in \mF_G \mid F \subseteq R\}$ as the set of faces of the braid arrangement contained in $R$ but not in any hyperplane of $\mA_G$. 
\end{definition}

The faces of the braid arrangement are in bijection with ordered partitions of $[n]$. This bijection is illustrated in Figure~\ref{fig:BraidFace}. Let $\Pi$ be an ordered partition of $[n]$. Then, the blocks of the partition indicate which coordinates are equal, that is, if $i$ and $j$ are in the same block of $\Pi$, then $x_i = x_j$ for all points $x$ of the face. Further, the relative ordering of the blocks indicates the relative ordering of the coordinates, that is, if $i$ is in a block before $j$, then $x_i < x_j$ for all points $x$ of the face. Given a face $F \in \mF_{\mB_n}$ we denote by $\Pi(F)$ the ordered partition labeling it and given an ordered partition of $[n]$, we denote by $F_{\Pi}$ the face it labels. 
\begin{example}
    For $n = 4$, the ordered partition $\{\{1,3\}, \{4\}, \{2\}\}$ corresponds to the face $\{(x_1, x_2, x_3, x_4) \in \mathbb{R}^4 \mid x_1 = x_3 < x_4 < x_2\}$.
\end{example}

\begin{figure}[h]
    \centering
    \includegraphics[width=0.8\linewidth]{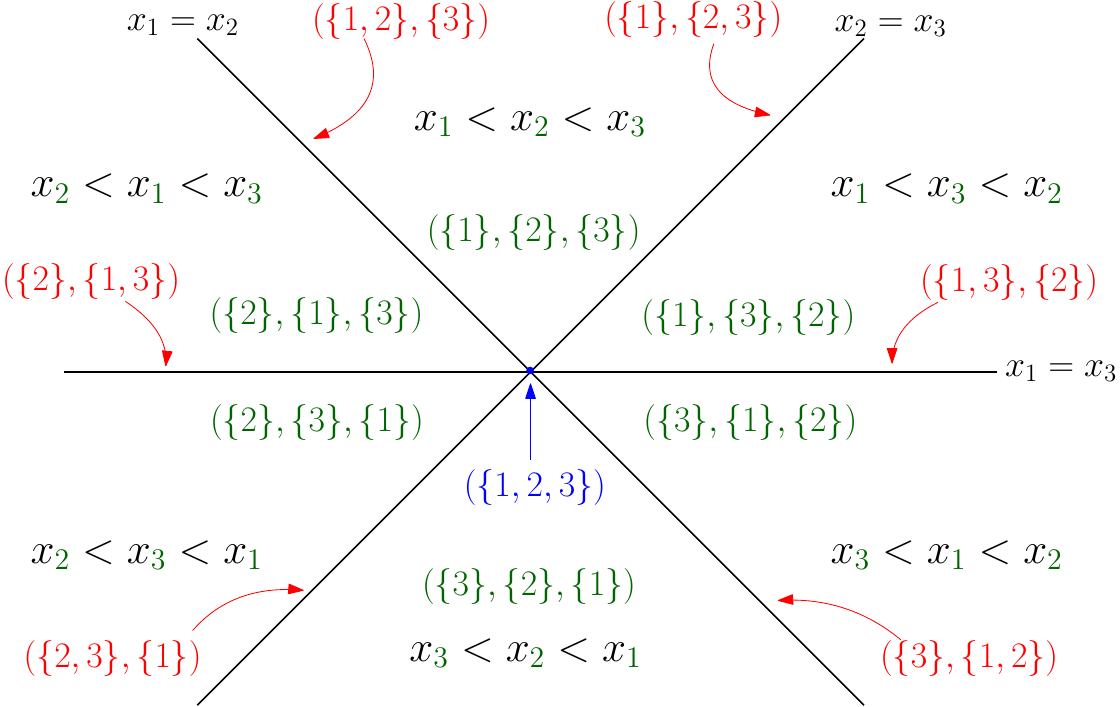}
    \caption{Bijection between faces of $\mB_3$ and ordered partitions of $[3]$.}
    \label{fig:BraidFace}
\end{figure}
\begin{lemma}\label{lem:graph_bijection}
    There is a bijection $\beta_R: \mU_{\bSG}(R) \ra \mF_{R}$ such that for an $\bS$-boxed tree $(T,B) \in \mU_{\textbf{S}}(R)$, $\dim(\beta_R(T,B)) = |B|$.  
\end{lemma}
\begin{proof}
    We define $\beta: \mU_{\bS_G}(n) \ra \mF_{G}$ as follows: 

    Let $(T, B) \in \mU_{\bS_G}(n)$. By Remark~\ref{rem:GBoxedTrees}, we can represent this as $T = v_1\ldots v_n$ and $B = \{(v_1, \ldots, v_{i_1}), \ldots, (v_{i_k + 1}, \ldots, v_n)\}$. We define $\Pi_{(T,B)} = (\{v_1, \ldots, v_{i_1}\},  \ldots, \{v_{i_k + 1}, \ldots, v_n\})$. Clearly, this is an ordered partition of $[n]$. We define $\beta(T, B) = F_{\Pi_{(T,B)}}$, that is, the face of $\mB_n$ labeled by the ordered partition $\Pi_{(T,B)}$. 

    Let $(T,B) \in \mU_{\bS_G}(n)$ and let $F:= \beta(T, B)$. By definition, $F \in \mF_{\mB_n}$. We wish to show that $F \in \mF_{G}$. Suppose $F \notin \mF_{G}$. Then, there exists a hyperplane $H = \{x_i - x_j = 0\}$ of $\mA_G$ such that $F \sse H$. Then, by the bijection between $\mF_{\mB_n}$ and ordered partitions of $[n]$, $i$ and $j$ will be in the same block of $\Pi(F) = \Pi_{(T, B)}$. But, by the definition of $\Pi_{(T, B)}$, $i$ and $j$ are in the same block if and only if they are in the same $\bS_G$-cadet sequence in $B$. By Lemma~\ref{graphcadet}, $\{i, j\} \notin E$. But, as $x_i - x_j$ is a hyperplane of $\mA_G$, $\{i, j\} \in E$, which is a contradiction. Hence $F \in \mF_{G}$, and $\beta$ is well defined. 

    Next, we show that $\beta$ is surjective. Let $F \in \mF_G$ and let $\Pi(F) = (D_1, \ldots, D_k)$ be the ordered partition labeling $F$. For $i \in [k]$, let $(v_{i_1}, \ldots, v_{i_{d_i}})$ be the ordering of the elements of $D_i$ in increasing order (where $|D_i| = d_i)$. Then, for $T = v_{1_1}\ldots v_{1_{d_1}} v_{2_1}\ldots v_{2_{d_2}}\ldots v_{k_1}\ldots v_{k_{d_k}}$ and $B = \{(v_{1_1},\ldots ,v_{1_{d_1}}), (v_{2_1},\ldots ,v_{2_{d_2}}),\ldots ,(v_{k_1},\ldots ,v_{k_{d_k}})\}$, we have $\beta(T, B) = F$. It remains to show that $(T, B) \in \mU_{\bS_G}(n)$. Suppose $(T, B) \notin \mU_{\bS_G}(n)$. Then, $\exists i \in [k]$ such that $(v_{i_1}, \ldots, v_{i_{d_i}})$ is not an $\bS_G$-cadet sequence. Note that by construction of $B$, the elements in this cadet sequence are in increasing order. Hence, by Lemma~\ref{graphcadet}, we have $\{v_{i_t}, v_{i_s}\} \in E$ for some $t, s \in [d_i]$. Now, as $v_{i_t}$ and $v_{i_s}$ are in the same block of $\Pi(F)$, $x_{v_{i_t}} = x_{v_{i_s}}$ for all points $x$ in $F$. Hence $F$ lies on the hyperplane $x_{v_{i_t}} - x_{v_{i_s}} = 0$ which is in $\mA_G$ as $\{v_{i_t}, v_{i_s}\} \in E$. This contradicts $F \in \mF_G$. Hence, $(T, B) \in \mU_{\bS_G}(n)$ such that $\beta(T, B) = F$. So $\beta$ is surjective. Further, as $(T, B)$ was uniquely determined from $\Pi(F)$ and hence from $F$, we have $\beta$ is injective. 

    Hence $\beta: \mU_{\bS_G}(n) \ra \mF_G$ is a bijection. Next, we show that for a region $R$ of $\mA_G$, $\beta_R := \beta|_{\mU_{\bS_G}(R)}$ is a bijection between $\mU_{\bS_G}(R)$ and $\mF_R$. This amounts to showing that $\forall (T, B) \in \mU_{\bS_G}(n)$, $\beta(T, B) \sse \Phi(T)$. 

    Let $(T, B) \in \mU_{\bS_G}(n)$ with $T = v_1\ldots v_n$ and $B = \{(v_1, \ldots, v_{i_1}), \ldots, (v_{i_k + 1}, \ldots, v_n)\}$. Let $F = \beta(T,B)$. We first show that $\Phi(T)$ is incident to $F$. 
    
    Now, by definition of $\Phi$, $x_{v_1} < \ldots < x_{v_n}$ for all $x \in \Phi(T)$. Further, by definition of $\beta$, $\Pi(F) = \Pi_{(T, B)} = (\{v_1, \ldots, v_{i_1}\},  \ldots, \{v_{i_k + 1}, \ldots, v_n\})$. Hence the weak inequalities that hold in $F$ hold in the closure of $\Phi(T)$. $F$ is contained in the closure of $\Phi(T)$. 

    Now, let $R$ be a region of $\mA_G$, and let $(T, B) \in \mU_{\bS_G}(R)$. Let $F = \beta(T, B)$. Then $F$ is in the interior of a region $R'$ of $\mA_G$, as if not, it would lie on a hyperplane of $\mA_G$, which contradicts $F \in \mF_G$. But, $F$ is contained in the closure of $\Phi(T)$, so $\Phi(T) \sse R'$. Hence $T \in T_{R'}$. As $(T,B) \in \mU_{\bS_G}(R)$, we have $T \in T_R$. Further, as $T_R \cap T_{R'} = \emptyset$ if $R \neq R'$, we have $R = R'$.  

    Hence $\beta_R := \beta\big|_{\mU_{\bS_G}(R)}: \mU_{\bS_G}(R) \ra \mF_R$ is a bijection. 

    Finally, it is easy to show that the dimension of a face of $\mB_n$ is equal to the number of blocks of the ordered partition labeling it. Further, by definition of $\beta$, the ordered partition labeling $\beta(T,B)$ has $|B|$ blocks. Hence, 
    $$\dim(\beta(T,B)) = |B|.$$ 
\end{proof}

\begin{corollary}[Theorem~\ref{generalcase} for Graphical Arrangements]\label{GAResult}
Let $G = ([n], E)$ be a graph and let $\mA_G$ be the corresponding graphical arrangement. Let $R$ be a region of $\mA_G$. Then, $$\sum_{(T,B) \in \mU_{\bS_G}(R)} (-1)^{n-|B|} = 1.$$
\end{corollary}

\begin{proof}
By the bijection established in Lemma \ref{lem:graph_bijection}, for any region $R$ of $\mathcal A_G$, we have
$$\sum_{(T,B) \in \mU_{\bS_G}(R)} (-1)^{n-|B|} = \sum_{F \in \mathcal F_R} (-1)^{n-\dim(F)}.$$
Since the faces in $\mathcal F_R$ partition $R$, the Euler characteristic of the region is $\sli_{F \in \mathcal F_R} (-1)^{n-\dim(F)}$. As the region is contractible, this equals 1.
\end{proof}
\smallskip

\subsection{General Case}\label{GenSec}\hfill\\

We now prove Theorem~\ref{generalcase}, that is, that Equation~\eqref{eqn:bernardi_region} holds for general deformations of the braid arrangement. 

Our proof follows similar lines to that of Corollary~\ref{GAResult}. For $m = \max\{|s| \mid s \in S_{a, b}, 1\leq a < b \leq n\}$, we show that the $\bS$-boxed trees are in bijection with certain faces of the $m$-Catalan arrangement, and then use the Euler characteristic to find the exact value of the signed sum. 

We first recall the bijection of the faces of the $m$-Catalan arrangement with some decorated plane trees as given in~\cite{bernardi2025bijectionsfacesbraidtypearrangements, Dlev}.

\begin{figure}[h]
    \centering
    \includegraphics[width=0.45\linewidth]{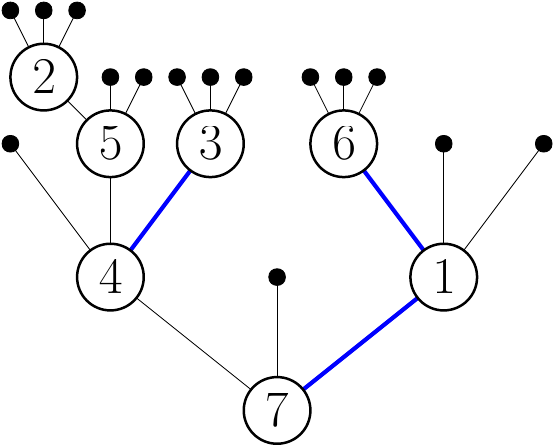}
    \caption{An example of a marked $(2, 7)$-tree.}
    \label{fig:marked_tree}
\end{figure}

\begin{definition}
    A \emph{marked $(m,n)$-tree} is a pair $(T, \mu)$ where $T \in \mT^{(m)}(n)$ and $\mu$ is a set of cadet edges of $T$ such that if an edge $e \in \mu$ is of the form $e = \{j, \schild{0}{j}\}$, then $j < \schild{0}{j}$. We refer to the edges in $\mu$ as the \emph{marked edges}. Further, we write $i \sim_{\mu} j$ if there is a path between $i$ and $j$ such that all the edges are in $\mu$. 

    We denote by $\overline{\mT^{(m)}(n)}$ the set of marked $(m,n)$-trees. An example of a marked $(2,7)$-tree is shown in Figure~\ref{fig:marked_tree}.
\end{definition}

\begin{theorem}[{\cites[Theorem 3.8]{bernardi2025bijectionsfacesbraidtypearrangements}[Theorem 14]{Dlev}}]\label{CatalanFacesBij}
    The faces of the $m$-Catalan arrangement are in bijection with the trees in $\overline{\mT^{(m)}(n)}$, where the bijection is given by: 
    \begin{align*}
        \Psi(T, \mu) := 
\left(\bigcap_{\substack{\{i,j\} \in \mu \\ i = \schild{s}{j} }} \{x_i - x_j = s\} \right) 
&\cap 
\left(\bigcap_{\substack{(i,j,s) \in \text{Triple}_n^m \\ i \not\sim_\mu j,\, i \prec_T \schild{s}{j}}} \{x_i - x_j < s\} \right) \\
&\cap 
\left(\bigcap_{\substack{(i,j,s) \in \text{Triple}_n^m \\ i \not\sim_\mu j,\, i \succeq_T \schild{s}{j}}} \{x_i - x_j > s\}\right),
    \end{align*}
    where $Triple_n^m = \{(i,j,s) \mid i, j \in [n], i \neq j,  \, s \in [0..m] \text{ such that } (s > 0 \text{ or } i>j)\}. $
\end{theorem}
Figure~\ref{fig:CatFace} illustrates the bijection for the $1$-Catalan arrangement.

\begin{remark}
    Note that by the bijection $\Psi$ defined above, for a marked tree $(T, \mu)$, if the edge $\{i, j\}$ is marked with $i = \schild{s}{j}$, then the face $\Psi(T,\mu)$ lies on the hyperplane $\{x_j + s = x_i\}$. 
\end{remark}

\begin{remark}
    For a face $F$ of the $m$-Catalan arrangement, let $(T, \mu) \in \overline{\mT^{(m)}(n)}$ be such that $\Psi(T, \mu) = F$. Then, it is clear from the definitions of $\Phi$ and $\Psi$ that the region of the $m$-Catalan arrangement given by $\Phi(T)$ is incident to $F$. 
\end{remark}

\begin{figure}
    \centering
    \includegraphics[width=\linewidth]{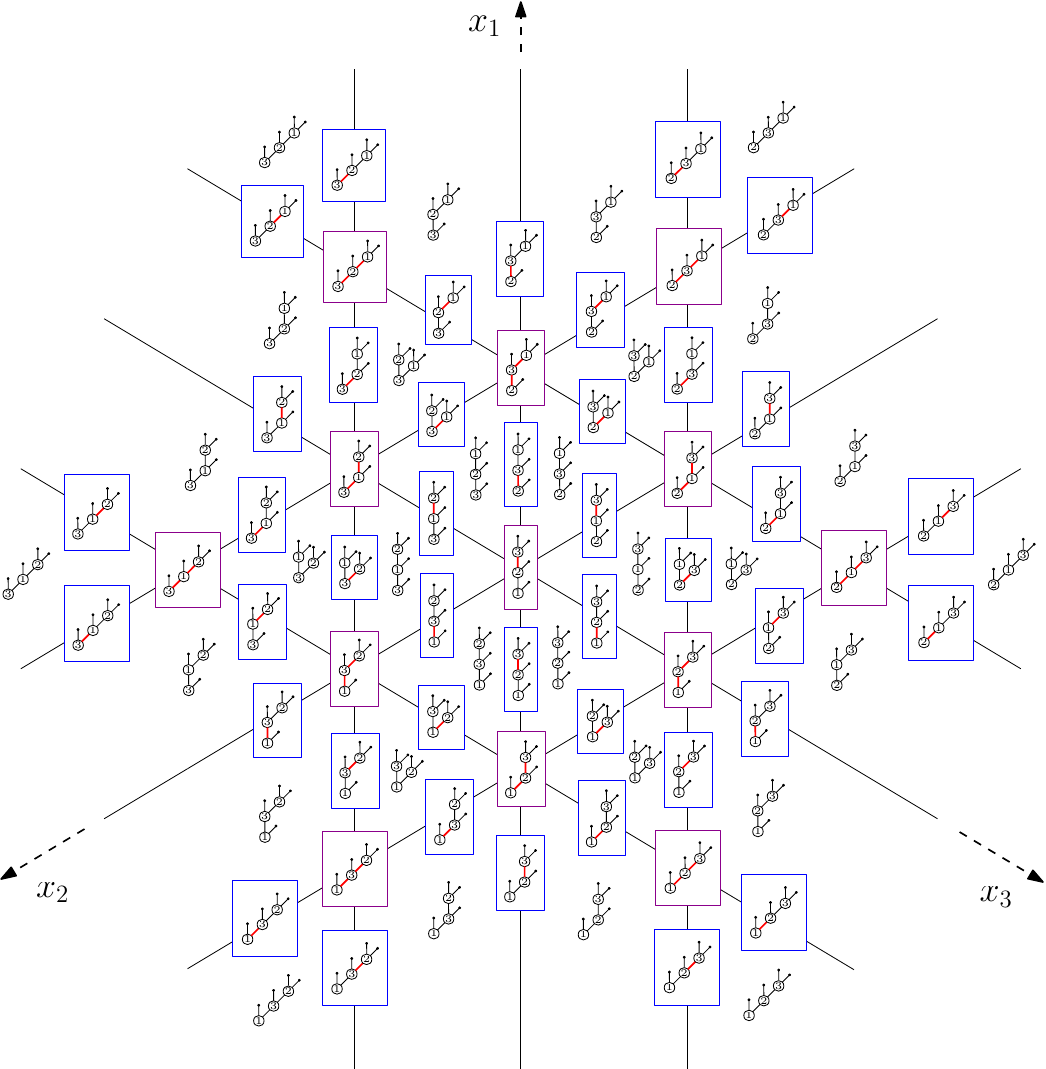}
    \caption{The 1-Catalan arrangement and the marked trees corresponding to each face. In the trees, the presence of a marked edge between nodes 1 and 2, 1 and 3, or 2 and 3 indicates that the corresponding face lies on a hyperplane involving the variables $x_1$ and $x_2$, $x_1$ and $x_3$, or $x_2$ and $x_3$, respectively.}
    \label{fig:CatFace}
\end{figure}

We now prove that Equation~\eqref{eqn:bernardi_region} holds for general deformations of the braid arrangement. 

\begin{definition}
Let $\mathcal{F}_n^m$ denote the set of faces of the $n$-dirmensional $m$-Catalan arrangement. For a given $\bS$-braid arrangement $\mathcal{A}_{\bS}$, define:
\[
\mathcal{F}_{\bS} := \{ F \in \mathcal{F}_n^m \mid \forall H \in \mA_{\bS}, \, F \not\subseteq H  \},
\]
that is, $\mathcal{F}_{\bS}$ is the set of faces of the $m$-Catalan arrangement that are not contained in any hyperplane of the $\bS$-braid arrangement.

Further, for a region $R$ of $\mA_{\bS}$, we define $\mF_{\bS}(R) := \{F \in \mF_{\bS} \mid F \subseteq R\}$.
\end{definition}

\begin{definition} Let $T$ be a tree in $\mT^{(m)}(n)$. We say $B$ is a \emph{boxing} of $T$ if $B$ is a set of cadet sequences partitioning the nodes of $T$ and furthermore, if an edge of the form $e = \{j, \schild{0}{j}\}$ is in a cadet sequence in $B$, then $j < \schild{0}{j}$.
    
We further define $\mU^{(m)}(n) := \{(T, B) \mid T \in \mT^{(m)}(n), B \text{ is a boxing of } T\}$.
\end{definition}

\begin{lemma}\label{lem:bijection_general}
There is a bijection $\beta_R: \mathcal{U}_{\bS}(R) \ra \mF_{\bS}(R)$ such that $\dim(\beta_R(T,B)) = |B|$. 
\end{lemma} 
\begin{proof} 
    We define $\alpha : \mU^{(m)}(n) \ra \overline{\mT^{(m)}(n)}$ as follows: 
    
    For $(T, B) \in \mU^{(m)}(n)$ a boxed tree, we define $\mu_B$ to be the set of edges of $T$ such that both endpoints are in the same box of $B$. Then, as boxes are cadet sequences, $\mu_B$ is a set of cadet edges. Further, if we have $e = \{u, v\} \in \mu_{B}$ with $v = \schild{0}{u}$, as $e$ is contained in a box, we have $v < \schild{0}{u}$. Hence $(T, \mu_B) \in \overline{\mT^{(m)}(n)}$. Clearly, $\al$ is a bijection.

    Then, $\beta_m := \Psi \circ \al : \mU^{(m)}(n) \ra \mF^m_n$ is a bijection by Theorem~\ref{CatalanFacesBij}. Let $\beta_{\bS} := \beta_m \big|_{\mU_{\bS}(n)}$. Then $\beta_{\bS}$ is injective. 

    \begin{adjustwidth}{2em}{0pt}
        \textbf{Claim:} $\beta_{\bS}(\mU_{\bS}(n)) = \mF_{\bS}$.

        Suppose $F \notin \mF_{\bS}$. Then, $F$ lies on a hyperplane $\{x_i - x_j = s\} \in \mA_{\bS}$. By the bijection in Theorem~\ref{CatalanFacesBij}, for $(T, \mu) = \Psi^{-1}(F)$, we have a path of marked edges $jv_1\ldots v_k i$ joining $i$ and $j$, with $v_{\ell} = \schild{s_{\ell}}{v_{\ell - 1}}$ for $\ell \in [k+1]$ (taking $v_0 = i$ and $v_{k+1} = j$) such that $\sli_{\ell = 1}^k s_{\ell} = s$. Then, for $(T, \beta_{\mu}) = \al^{-1}(T, \mu)$, we have $i$ and $j$ are in the same box. But, the cadet sequence containing $i$ and $j$ will contain $(j, v_1, \ldots, v_k, i)$ which is not an $\bS$-cadet sequence as $s \in S_{i,j}^-$. Hence $(T, \beta_{\mu}) \notin \mU_{\bS}(n)$.

        Let $F \in \mF_{\bS}$, let $(T, \mu) = \Psi^{-1}(F)$ and let $(T, B_{\mu}) = \al^{-1}(T, B)$. We wish to show that $(T, B_{\mu})$ is an $\bS$-boxed tree. Let $(v_1, \ldots, v_k) \in B_{\mu}$, and let $v_{\ell+1} = \schild{s_\ell}{v_\ell}$ for $\ell \in [k-1]$. 

        Now, let $1 \leq i < j \leq n$. To show that $(v_1, \ldots, v_k)$ is an $\bS$-cadet sequence, we need to show that $\sli_{\ell = i}^{j-1} \lsib(v_{\ell + 1}) =\sli_{\ell = i}^{j-1} s_{\ell} \notin S_{v_i, v_j}^-$. Note that for $\ell \in [i; j-1]$, the edge $\{v_{\ell}, v_{\ell + 1}\}$ is a marked edge. Hence $F$ lies on the hyperplane $\{x_{v_{\ell + 1}} - x_{v_{\ell}} = s_{\ell}\}.$ As $F \in \mF_{\bS}$, we must have $\{x_{v_{\ell + 1}} - x_{v_{\ell}} = s_{\ell}\} \notin \mA_{\bS}$. Hence, $\bigcap\limits_{\ell = i}^{j-1} \{x_{v_{\ell + 1}} - x_{v_{\ell}} = s_{\ell}\} = \left\{x_j - x_i = \sli_{\ell = i }^{j-1} s_{\ell}\right\} \notin \mA_{\bS}$. 

        Clearly, if $\sli_{\ell = i}^{j-1} s_{\ell} \neq 0$, we have $\sli_{\ell = i}^{j-1} s_{\ell} \notin S_{v_i, v_j}^-$. Further, if $\sli_{\ell = i}^{j-1} s_{\ell} = 0$, as all the edges are marked, we have $v_i < \ldots < v_j$ and hence $\sli_{\ell = i}^{j-1} s_{\ell} \notin S_{v_i, v_j}^-$. 

        Hence $(v_1, \ldots, v_k)$ is an $\bS$-cadet sequence, and therefore $(T, B_{\mu})$ is an $\bS$-boxed tree. 

        Hence our claim holds. 
    \end{adjustwidth}

     Now, let $R$ be a region of $\mA_G$, and let $(T, B) \in \mU_{\bS}(R)$. Let $F = \beta(T, B)$. Then $F$ is in the interior of a region $R'$ of $\mA_G$, as if not, it would lie on a hyperplane of $\mA_G$, which contradicts $F \in \mF_{\bS}$. But, $F$ is contained in the closure of $\Phi(T)$, so $\Phi(T) \sse R'$. Hence $T \in T_{R'}$. As $(T,B) \in \mU_{\bS}(R)$, we have $T \in T_R$. Further, as $T_R \cap T_{R'} = \emptyset$ if $R \neq R'$, we have $R = R'$.  

    Hence, $\beta_R = \beta \big|_{\mU_{\bS}(R)} : \mU_{\bS}(R) \ra \mF_{R}$ is a bijection.  

    Finally, for a marked tree $(T, \mu)$, the dimension of $\Psi(T, \mu)$ is the number of $\mu$-connected components of $T$, which by definition of $\al$ is the number of boxes in $B$. Hence, $\dim(\beta(T, B)) = |B|$. 
\end{proof}

Theorem~\ref{generalcase} is then a consequence of Lemma~\ref{lem:bijection_general} together with an Euler characteristic argument. 

\begin{proof}[Proof of Theorem~\ref{generalcase}]
By Lemma~\ref{lem:bijection_general}, the boxed trees $\mathcal{U}_{\bS}(R)$ are in bijection with faces $\mathcal{F}_{\bS}(R)$ of the $m$-Catalan arrangement that lie in the region $R$ and avoid all hyperplanes in $\mathcal{A}_{\bS}$. The dimension of a face equals the number of boxes in the corresponding boxed tree, so:
\[
\sum_{(T,B) \in \mathcal{U}_{\bS}(R)} (-1)^{n - |B|} = \sum_{F \in \mathcal{F}_{\bS}(R)} (-1)^{n - \dim(F)}.
\]
Since the faces in $\mathcal{F}_{\bS}(R)$ partition the region $R$, the Euler characteristic of the region is $\sum\limits_{F \in \mathcal{F}_{\bS}(R)} (-1)^{n - \dim(F)}$. As the region is contractible, this equals $1$. 
\end{proof}

\section{The Contribution of a Tree} \label{sec:contribution}

In the previous sections, we considered the restriction of the Bernardi formula to a region and called it the \emph{contribution of a region}. We can similarly restrict the Bernardi formula to a tree and consider the \emph{contribution of a tree}. Bisain and Hanson~\cite{Bisain_2021} showed that the contribution of a tree is either $0$ or $\pm 1$ and further gave an algorithm to compute it. In this section, we use our new geometric insight to further understand the contribution of a tree. 

\begin{definition}
    Let $T \in \mT^{(m)}(n)$. We define the \emph{contribution of the tree $T$}, denoted $w(T)$ by 
    $$w(T) = \sum_{B \in \mU_{\bS}(T)} (-1)^{n-|B|},$$
    where $\mU_S(T)$ is the set of $\bS$-boxings of $T$.
\end{definition}

Let $\mF_T = \beta(\{(T, B) \mid B \in \mU_{\bS}(T)\})$. Then, by Lemma~\ref{lem:bijection_general}, $w(T) = \sum\limits_{B \in \mU_{\bS}(T)} (-1)^{n-|B|} = \sli_{F \in \mF_T} (-1)^{n -\dim(F)}$. To better understand the contribution of a tree, we first characterize the faces that lie in $\mF_T$. 

\begin{definition}
    Let $R$ be a region of a hyperplane arrangement $\mA_{\bS}$. We say a hyperplane $H_{a,b,s}$ of the $m$-Catalan arrangement is
    \begin{itemize}
        \item \emph{face-supporting} if $H_{a,b,s}$ has a non-empty intersection with the closure of $R$, and
        \item \emph{facet-supporting} if the intersection of $H_{a,b,s}$ with the closure of $R$ has dimension $n-1$. 
    \end{itemize}
    We further denote the set of face-supporting hyperplanes of $R$ by $\mH_R$.  
\end{definition}

\begin{definition}
    Let $T \in \mT^{(m)}(n)$ and $R_T = \Phi(T)$ be the corresponding region of the $m$-Catalan arrangement. We define the set $\mH^{\text{nsep}}(T)$ to be the collection of the face-supporting hyperplanes $H_{a,b,s} \in \mH_{R_T}$ such that 
    \begin{itemize}
        \item If $s > 0$, we have $x_a - x_b < s$ for all $x \in R_T$
        \item If $s = 0$ and $a<b$, $x_a - x_b > 0$ for all $x \in R_T$. 
    \end{itemize}
    and the set $\mH^{\mA}(T):= \{H \in \mH_{R_T} \mid H \in \mA_{\bS}\}$. 
    
    Further, let $\mH_{\bS}(T) := \mH^{\text{nsep}}(T) \cup \mH^{\mA}(T)$.
\end{definition}

\begin{figure}[ht]
        \centering
        \includegraphics[width=0.7\linewidth]{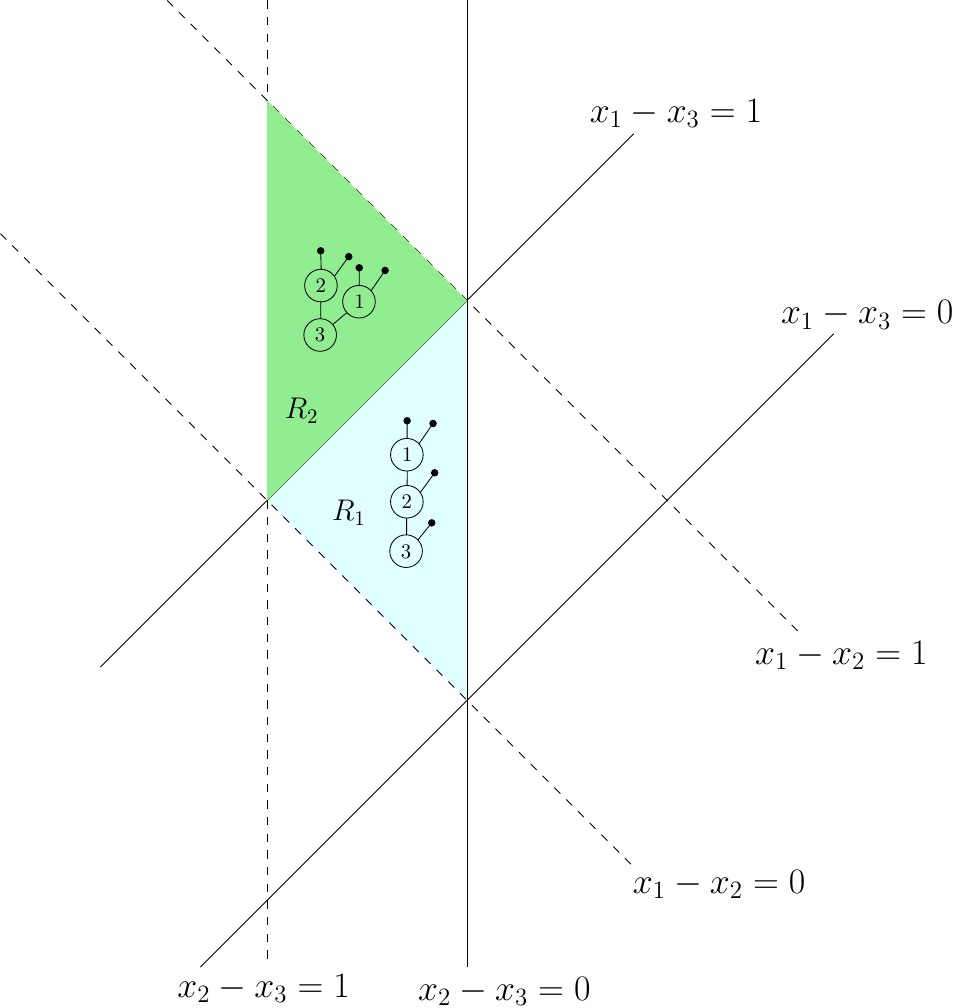}
        \caption{A hyperplane arrangement $\mA_{\bS}$. The hyperplanes represented by solid lines are in the arrangement while those represented by dashed lines are not in the arrangement.}
        \label{fig:CoTExample}
\end{figure}

\begin{example}\label{HSExample}
    For the hyperplane arrangement shown in Figure~\ref{fig:CoTExample}, let $T_1$ be the tree labeling $R_1$ and $T_2$ be the tree labeling $R_2$. 

    Note that $\mH_{R_1} = \{H_{1,2,0}, H_{1,2,1}, H_{1,3,0}, H_{1,3,1}, H_{2,3,0}, H_{2,3,1}\}$. Further, $\mH^{\text{nsep}}(T_1) = \mH_{R_1}$ and $\mH^{\mA}(T_1) = \{ H_{1,3,0}, H_{1,3,1}, H_{2,3,0}\}$. 

    Similarly, $\mH_{R_2} = \{H_{1,2,0}, H_{1,2,1}, H_{1,3,1}, H_{2,3,0}, H_{2,3,1}\}$. As $x_1 - x_3 > 1$ for all $x \in R_2$, $\mH^{\text{nsep}}(T_2) = \{H_{1,2,0}, H_{1,2,1},H_{2,3,0}, H_{2,3,1}\}$ and $\mH^{\mA}(T_2) = \{H_{1,3,1}, H_{2,3,0}\}$.
\end{example}

\begin{remark}
    Let $R$ be a region of the $m$-Catalan arrangement. Note that $\mH^{\text{nsep}}(T)$ consists of the hyperplanes $H \in \mH_R$ for which $R$ and the fundamental alcove are on the same side of $H$. For any region, there will be at least one facet-supporting hyperplane such that $R$ and the fundamental alcove are on the same side of $H$. Hence, $\mH^{\text{nsep}}(T)$ contains at least one facet-supporting hyperplane. 
\end{remark}

Then, we have the following characterization of the faces in $\mF_T$. 
\begin{lemma}\label{lem:CharOfF_TWRTHyperplanes}
    Let $T\in T^{(m)}(n)$ and let $R_T$ be the corresponding region of the $m$-Catalan arrangement. Then, for $F \in \mF_{R_T}$, we have $F \in \mF_T$ if and only if for any hyperplane $H$ of the $m$-Catalan arrangement such that $F$ lies on $H$, we have $H \notin \mH_{\bS}(T)$. 
\end{lemma}
\begin{proof}
    Let $F \in \mF_{R_T}$ and suppose $F$ lies on a hyperplane $H_{a,b,s} \in \mH_{R_T}$. We wish to show $F \in \mF_T$ if and only if $H_{a,b,s} \notin \mH_{\bS}(T)$. 
    
    Clearly, from Lemma~\ref{lem:bijection_general}, $F \in \mF_T$ if and only if it does not lie on any hyperplane of $\mA_{\bS}$, that is, if and only if $H_{a,b,s} \notin \mH^{\mA}(T)$. 

    Now, suppose $F \in \mF_T$. As $F$ lies on $H_{a,b,s}$, by definition of $\beta$, we have $a$ and $b$ are in the same $\bS$-cadet sequence. If $s = 0$, we can assume without loss of generality that $a < b$. Then, all the nodes on the path between $a$ and $b$ will be $0$-children, and hence $a$ is before $b$ in the $\bS$-cadet sequence. Then by the definition of $\Phi$, for all $x \in R_T$, $x_a - x_b < 0$. If $s >0$, by definition of $\Psi$, we have $a$ is after $b$ in the $\bS$-cadet sequence and hence $x_a - x_b > s$ for all $x \in R_T$. Hence $H_{a,b,s} \notin \mH^{\text{nsep}}(T)$. 

    Hence the result holds. 
\end{proof}

To better understand these faces, we use a transformation on the region. 

\begin{lemma}\label{lem:transformation}
    Let $R$ be a relatively bounded region of the $m$-Catalan arrangement. Then, there is an affine transformation from $R$ to the region $y_1 < y_2 < \ldots < y_n < y_1 + 1$. 
\end{lemma}
\begin{proof}
    Let $p = (a_1, \ldots, a_n)$ be an arbitrary point in the interior of $R$. For all $i \in [n]$, let $b_i = a_i - \lfloor a_i \rfloor$. Let $\sigma$ be the permutation such that $b_{\sigma(1)} < b_{\sigma(2)} < \ldots < b_{\sigma(n)}$. 

    Then, for $y_i = x_{\sigma(i)} - \lfloor a_{\sigma(i)}\rfloor$, we have $y_1 < y_2 < \ldots < y_n < y_1 + 1$. This will hold for all $x \in R$ as the transformation is affine and preserves the relative ordering structure of the original region. 
\end{proof}

\begin{remark}
    Note that by the above transformation, any hyperplane of the form $x_i - x_j = s$ in an $\bS$-braid arrangement is transformed to a hyperplane of the form $y_{\sigma^{-1}(i)} - y_{\sigma^{-1}(j)} = s + \lfloor a_i \rfloor - \lfloor a_j \rfloor$. We denote this hyperplane by $H_{i,j,s}^Y$. 

    Further, while the transformation in Lemma~\ref{lem:transformation} is for relatively bounded regions, as the contribution of a tree is independent of $m$, we may take $m$ large enough such that the region labeled by the tree (up to additional right leaves) is relatively bounded. 
\end{remark}

\begin{definition}\label{def:tau}
    Let $T \in T^{(m)}(n)$ be a tree and $R_T$ be the corresponding region of the $m$-Catalan arrangement. Let $F$ be a face incident to $R_T$. We define $$\tau(F) : = \{i \in [n] \mid F \text{ lies on $H_i$}\},$$ where for $i \in [n-1]$, $H_i = \{y_i = y_{i+1}\}$ and $H_n = \{y_n = y_1 + 1\}$. 
\end{definition}

It is easy to see that for a face $F$ of the $m$-Catalan arrangement, $|\tau(F)| = n - \dim(F)$. 

\begin{example}
    For $n=4$ and a region $R_T$ as described above, suppose a face $F$ is characterized in the transformed coordinates by $y_1 = y_2 < y_3 = y_4 < y_1 + 1$. Then:
    \begin{itemize}
        \item $F$ lies on $H_1: y_1 = y_2$, so $1 \in \tau(F)$
        \item $F$ does not lie on $H_2: y_2 = y_3$, so $2 \notin \tau(F)$
        \item $F$ lies on $H_3: y_3 = y_4$, so $3 \in \tau(F)$
        \item $F$ does not lie on $H_4: y_4 = y_1 + 1$, so $4 \notin \tau(F)$
    \end{itemize}
    Thus $\tau(F) = \{1, 3\}$ and $\dim(F) = 2$.
\end{example}

\begin{lemma}
    Let $R$ be a region of the $m$-Catalan arrangement and let $H$ be a face-supporting hyperplane. Let $F(H)$ be the face of formed by intersection $H$ and the closure of $R$. Then if $H^Y = \{y_a - y_b = s\}$ is the transformed form of $H$ (with $s \geq 0$ and $a < b$ if $s = 0$), we have 
    $$\tau(F(H)) = \begin{cases}
        \{a, \ldots, b-1\} & \text{ if } a<b \\
        [n] \setminus \{b, \ldots, a-1\} & \text{ if } a>b
    \end{cases}. $$
\end{lemma}
\begin{proof}
    Note that any face supporting hyperplane of $R$ after the transformation will be of the form $y_a - y_b = 0$ or $y_a - y_b = 1$ where $a >b$. 
    \begin{adjustwidth}{2em}{0pt}
        \noindent \textbf{Case 1}: $H^Y = \{y_a - y_b = 0\}$. 

        Without loss of generality, let $a < b$. Then, under the transformation, $F(H)$ is the face $y_1 < \ldots < y_a = \ldots = y_b < \ldots < y_n < y_1+1$. Here, $\tau(F(H)) = \{a, \ldots, b-1\}$. 

        \noindent \textbf{Case 2}: $H^Y = \{y_a - y_b = 1\}$ where $a > b$. 

        Under the transformation, $F(H)$ is the face $y_1 = \ldots =  y_b < \ldots < y_a = \ldots = y_n = y_1 + 1$. Here, $\tau(F(H)) = [n] \setminus \{b, \ldots, a-1\}$. 
    \end{adjustwidth}
\end{proof}

We call $\tau(F(H))$ the \emph{interval tuple} associated to $H$ and denote it $I(H)$. 

\begin{proposition}\label{lem:DefOfPT}
    Let $T \in \mT^{(m)}(n)$ and let $R_T = \Phi(T)$. Then, the map $\tau$ as defined in Definition~\ref{def:tau} defines a bijection between $\mF_T$ and the set $\mP_T = \{ D \sse [n] \mid D\neq [n], D\nsupseteq I(H) \text{ for all } H \in \mH_{\bS}(T)\}$. 
\end{proposition}
\begin{proof}
    It is easy to see that $\tau$ is well defined and injective. Let $F$ be a face of $R_T$. 

    Now, suppose $F \notin \mF_T$. Then $F$ lies on some hyperplane $H$ of the $m$-Catalan arrangement, and we have $F \sse F(H)$. Hence, $\tau(F) \supseteq \tau(F(H)) = I(H)$. By Lemma~\ref{lem:CharOfF_TWRTHyperplanes}, $H \in \mH_{\bS}(T)$. Hence $\tau(F) \notin \mP_T$. 

    Conversely, suppose $\tau(F) \notin \mP_T$. Then, $\tau(F) \supseteq I(H)$ for some $H \in \mH_{\bS}(T)$. But then $F \sse F(H)$, that is, $F$ lies on the hyperplane $H$. As $H \in \mH_{\bS}(T)$, by Lemma~\ref{lem:CharOfF_TWRTHyperplanes}, $F \notin \mF_T$. 
\end{proof}

The following is now a straightforward consequence of the above proposition. 
\begin{corollary}\label{Cor:ContriOfTreePT}
    Let $T \in \mT^{(m)}(n)$ and let $R = \Phi(T)$. Then,
    $$w(T) = \sli_{D \in \mP_T} (-1)^{|D|}.$$
\end{corollary}

\begin{remark}
    Note that in the definition of $\mP_T$, it is enough to consider the minimal (under inclusion) interval tuples associated to $\mH_{\bS}(T)$. That is, for $\mI_{\bS} = \{I(H) \mid H \in \mH_{\bS}(T), I(H) \not\supset I(H') \text{ for all } H' \in \mH_{\bS}(T)\}$, we have $\mP_T = \{D \sse [n] \mid D \neq [n], I \nsubseteq D \text{ for all } I \in \mI_{\bS}\}.$
\end{remark}

\begin{example}
    Consider the hyperplane arrangement shown in Figure~\ref{fig:CoTExample}. From Example~\ref{HSExample}, we know that $\mH_{\bS}(T_1) = \mH_{R_1}$, that is, $\mH_{\bS}(T_1)$ consists of all the face-supporting hyperplanes of $R_1$. As this includes all the facet-supporting hyperplanes of $R_1$, we have $\mI_{\bS}$ contains all the singleton sets. Hence, $\mP_{T_1} = \{\emptyset\}$, and by Corollary~\ref{Cor:ContriOfTreePT}, we have $w(T_1) = 1$. 

    In fact, for any region such that $\mH_{\bS}(T)$ contains all the facet-supporting hyperplanes, we will similarly have $w(T) = 1$. As $R_2$ is such a region, we can conclude that $w(T_2) = 1$. 
\end{example}

We have seen that $\mH^{\text{nsep}}(T)$ contains at least one facet-supporting hyperplane. The interval tuple corresponding to such a hyperplane is of cardinality $1$, hence $\mI_{\bS}$ contains at least one singleton set. 

We now consider the intervals in $\mI_{\bS}$ as intervals on the circle. Their union will consist of disjoint intervals $J_1, \ldots, J_t$. Note that as $\mI_{\bS}$ consists of minimal intervals and at least one of them is a singleton, $t \geq 2$. Clearly, for $I \in \mI_{\bS}$, we have $I \sse J_i$ for some $i$. 

\begin{definition}
    Let $G = ([n], E)$ be a graph where $E = \{(i, i+1) \mid  \{i, i+1\} \sse I \text{ for some } I \in \mI_{\bS}\}$ and the vertices are considered modulo $n$. Let $V_{I} = \{i \in [n] \mid i \in I \text{ for some } I \in \mI_{\bS}\}$. We define the \emph{circular connected components} of $\mI_{\bS}$ to be the vertex sets of the connected components of $G[V_{I}]$. 

    Further, for a circular connected component $J$ of $\mI_{\bS}$, we define $\mJ = \{I \in \mI_{\bS} \mid I \sse J\}$. 
\end{definition}

\begin{example}\label{CCCexample}
    For $n = 8$, suppose $\mI_{\bS} = \{\{1,2\}, \{2,3\}, \{4,5\}, \{6\}, \{8,1\}\}$. Then, we have $J_1 = \{8,1,2,3\}$, $J_2 = \{4,5\}$ and $J_3 = \{6\}$. This is illustrated in Figure~\ref{fig:CCCexample}. 

    \begin{figure}[h]
        \centering
        \includegraphics[width=0.35\linewidth]{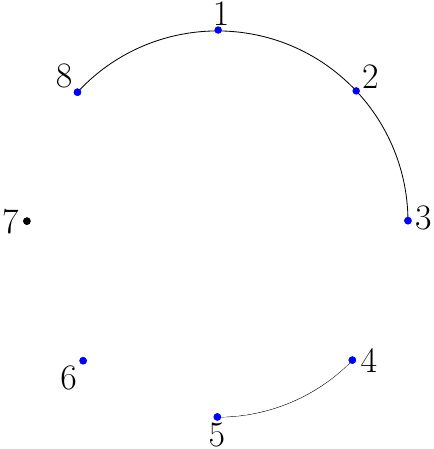}
        \caption{The corresponding graph for Example~\ref{CCCexample} with vertices in $V_I$ in blue.}
        \label{fig:CCCexample}
    \end{figure}
\end{example}

Clearly, if $J_1, \ldots, J_t$ are the circular connected components of $\mI_{\bS}$, then $\mJ_1, \ldots, \mJ_t$ partition $\mI_{\bS}$. 

\begin{definition}
    Let $J \subsetneq [n]$ be an interval and let $\mJ = \{I_1, \ldots, I_k\}$ be a set of minimal subintervals of $J$ such that $\bigcup\limits_{i = 1}^k I_i = J$. We define $$w(\mJ) = \sli_{\substack{D \sse J \\ I \nsse D \forall I \in \mJ}} (-1)^{|D|}.$$
\end{definition}

Then, 
\begin{proposition}
    Let $T \in \mT^{(m)}(n)$ and let $R = \Phi(T)$. Let $J_1, \ldots, J_t$ be the circular connected components of $\mI_{\bS}$. Let $S = [n] \setminus \bigcup\limits_{i = 1}^t J_i$. Then, 
    $$w(T) = \begin{cases}
        \prod\limits_{i = 1}^t w(\mJ_i) & \text{ if } S = \emptyset \\
        0 &\text{ if } S \neq \emptyset. 
    \end{cases}$$
\end{proposition}
\begin{proof}
    Let $D \in \mP_T$, let $D_i = \{a \in D \mid a \in J_i\}$ and let $D^* = \{a \in D \mid a \in S\}$. Clearly, the $D_i$'s and $D^*$ partition $D$. 

    Further, 
    \begin{align*}
        w(T) = \sli_{D \in \mP_T} (-1)^{|D|} &= \sli_{D \in \mP_T}(-1)^{|D^*|} \prod\limits_{i = 1}^t (-1)^{|D_i|} \\&= \left(\sli_{D^* \sse S} (-1)^{|D^*|}\right)\left(\prod\limits_{i = 1}^t \sli_{\substack{D_i \sse J_i\\ I \nsse D_i \forall I \in \mJ_i}} (-1)^{|D_i|}\right) \\
        &= \sli_{D^{*} \sse S} (-1)^{|D^*|}\left(\prod\limits_{i = 1}^t w(\mJ_i)\right).
    \end{align*}
    Now, $\sli_{D^* \sse S} (-1)^{|D^*|} = \begin{cases}
        1 &\text{ if } S = \emptyset \\
        0 &\text{ if } S \neq \emptyset. 
    \end{cases}$ The result follows.  
\end{proof}

\begin{remark}
    It is easy to see that for $\mJ = \{J\}$, we have $w(\mJ) = (-1)^{|J| + 1}$. However, the value of $w(\mJ)$ is more complicated to determine in general.

    Let $F = \tau^{-1}(J)$ and suppose $\mJ = \{I_1, \ldots, I_k\}$. Let $F_i = \tau^{-1}(I_i)$. Then, geometrically, $w(\mJ)$ is the sum of $(-1)^{n - \dim(F')}$ over faces $F'$ of $R$ containing $F$ that are not contained in any of the faces $F_i$.
\end{remark}

\section{Acknowledgments}

We owe many thanks to Olivier Bernardi for proposing this problem and his invaluable guidance throughout this project. We thank Nathan Williams for reviewing our paper and giving us valuable feedback. We also thank the PRIMES organizers for their continued support throughout the project.

\printbibliography

\end{document}